\documentclass[3p,times]{elsarticle}

\usepackage[centertags]{amsmath}
\usepackage{amssymb}
\usepackage{amsthm}
\usepackage{hyperref}
\hypersetup{
    colorlinks=true,
    linkcolor=blue,
    filecolor=magenta,      
    urlcolor=cyan,
    pdftitle={Overleaf Example},
    pdfpagemode=FullScreen,
    }
    
    \usepackage{soul}
    \usepackage{color}
    \newcommand{\commentTN}[1]{\textcolor{blue}{#1}}
    \newcommand{\commentSG}[1]{\textcolor{red}{#1}}
    \newcommand{\commentYL}[1]{\textcolor{green}{#1}}
    \newcommand{\commentOUT}[1]{}%{\textcolor{red}{#1}}
    
\newcommand{\bbR}{\mathbb{R}}
\newcommand{\bbC}{\mathbb{C}}
\newcommand{\bbZ}{\mathbb{Z}}

\newcommand{\norm}[1]{\left\Vert#1\right\Vert}
\newcommand{\abs}[1]{\left\vert#1\right\vert}
\newcommand{\set}[1]{\left\{#1\right\}}
\newcommand{\Real}{\mathbb R}
\newcommand{\eps}{\varepsilon}
\newcommand{\To}{\longrightarrow}
\newcommand{\BX}{\mathbf{B}(\xx)}
\newcommand{\A}{\mathcal{A}}

\newcommand\varteta{\boldsymbol{\vartheta}}
\newcommand\teta{\boldsymbol{\theta}}
\newcommand\eeta{\boldsymbol{\eta}}

\newcommand\Teta{\boldsymbol{\Theta}}
\newcommand\xx{\boldsymbol{x}}

\newcommand{\bbS}{\mathcal{S}}
\newcommand{\bbSx}{\mathcal{S}_{\xx}}
\newcommand{\bbT}{\mathcal{T}}
\newcommand{\bbTx}{\mathcal{T}_{\xx}}
\newcommand{\bbI}{\mathbb{I}}
\newcommand{\bbX}{\mathbb{X}}
\newcommand{\bbH}{\mathbb{H}}
\newcommand{\bbE}{\mathbb{E}}
\newcommand{\calH}{\mathcal{H}}
\newcommand{\scpr}[2]{\left\langle#1|#2\right\rangle}
\newcommand{\calR}{\mathcal{R}}
\newcommand{\frakI}{\mathfrak{I}}
\newcommand{\frakK}{\mathfrak{K}}

\newcommand{\bbb}{\mathbf{b}}
\newcommand{\bbe}{\mathbf{e}}
\newcommand{\aalpha}{\boldsymbol{\alpha}}

\newcommand{\kerS}{{\rm ker\,}\bbS}
\newcommand{\OkerS}{({\rm ker\,}\bbS)^\bot}
\newcommand{\kS}{\mathfrak{S}}
\newcommand{\OkS}{\mathfrak{S}^\bot}

\newcommand{\Spec}{{\rm Spec\,}}

\newcommand{\pr}{\mbox{\sf P}}
\newcommand{\ex}{{\bf\sf E\,}}               %% expectation
\newcommand{\var}{\mbox{\sf Var}}
\newcommand{\sd}{\mbox{\sf sd}}
\newcommand{\cv}{\mbox{\sf cv}}
\newcommand{\pdf}{\mbox{\sf pdf}}

\newcommand{\mapsfrom}{\mathrel{\reflectbox{\ensuremath{\mapsto}}}}

\def\bsk{\bigskip\noindent}
\def\ssk{\medskip\noindent}
\def\hsk{\smallskip\noindent}
\newcommand{\evt}{{\cal E}}             %% event list
\newcommand{\bS}{{\bf S}}               %% state space
\newcommand{\bs}{{\bf s}}               %%
\newcommand{\bA}{{\bf A}}               %% event set
\newcommand{\bB}{{\bf B}}
\newcommand{\bI}{{\bf I}}
\newcommand{\bZ}{{\mathbb Z}}
\newcommand{\bu}{{\bf u}}               %% bold u
\newcommand{\bn}{{\bf n}}               %% bold n (vector)
\newcommand{\bb}{{\bf b}}               %% bold b (vector)
\newcommand{\bx}{{\bf \xx}}               %% bold \xx (vector)
\newcommand{\bd}{{\bf d}}
\newcommand{\bk}{{\bf k}}
\newcommand{\bg}{{\bf g}}
\newcommand{\bh}{{\bf h}}
\newcommand{\bp}{{\bf p}}
\newcommand{\be}{{\bf e}}
\newcommand{\by}{{\bf y}}               %% bold y (vector)
\newcommand{\bz}{{\bf z}}               %% bold z (vector)
\newcommand{\bw}{{\bf w}}
\newcommand{\rk}{{\bf r}}
\newcommand{\bzm}{{\bf Z}_+^m}

\newcommand{\N}{{\mathcal N}}
\newcommand{\calL}{{\mathcal L}} 
\newcommand{\calO}{{\mathcal O}}
\newcommand{\calm}{{\mathcal M}}
\newcommand{\calb}{{\mathcal B}}
\newcommand{\calc}{{\mathcal C}}
\newcommand{\calp}{{\mathcal P}}
\newcommand{\calg}{{\mathcal G}}
\newcommand{\calk}{{\mathcal K}}
\newcommand{\cals}{{\mathcal S}}
\newcommand{\calI}{{\mathcal I}}
\newcommand{\calJ}{{\mathcal J}}
\newcommand{\calstoi}{{\mathcal S}_{>i}}
\newcommand{\calsjto}{{\mathcal S}_{j>}}
\newcommand{\calsito}{{\mathcal S}_{i>}}

\newcommand{\mtp}{{\mathfrak P}}

\newcommand{\cala}{{\mathcal A}}
\newcommand{\calf}{{\mathcal F}}
\newcommand{\caln}{{\mathcal N}}

\newcommand{\al}{\alpha}                %%
\newcommand{\unalpha}{\underline{\alpha}}                %%
\newcommand{\ovalpha}{\overline{\alpha}}                %%

\newcommand{\bet}{\beta}                %%
\newcommand{\g}{\lambda}                %%
\newcommand{\ga}{\gamma}                %% abbreviated
\newcommand{\dt}{\beta}                %% greek letters
\newcommand{\Dt}{\beta}                %% greek letters
\newcommand{\la}{\lambda}               %%
\newcommand{\lam}{\lambda}               %%
\newcommand{\Lam}{\Lambda}               %%
\newcommand{\sig}{\sigma}               %%
\newcommand{\s}{\sigma}               %%
\newcommand{\om}{\omega}                %%
\newcommand{\Om}{\Omega}   

\newcommand{\ra}{\rightarrow}           %%
\newcommand{\Ra}{\Rightarrow}           %% arrows
\newcommand{\imp}{\Rightarrow}           %% arrows
\newcommand{\callRa}{\calleftrightarrow}      %%
\newcommand{\subs}{\subseteq}           %% subset or equal to
\newcommand{\stle}{\le_{\rm st}}        %% stochastically less than
\newcommand{\abk}{{\bf (a,b,k)}}

\newcommand{\hti}{{\tilde h}}
\newcommand{\tL}{{\tilde L}}
\newcommand{\wh}{\widehat}
\newcommand{\rinv}{R^{\mbox{\small -\tiny 1}}}   %% a nicer R^{-1}
\newcommand{\pre}{\preceq}
\newcommand{\Bbar}{\overline{B}}
\newcommand{\ub}{\underline{b}}
\newcommand{\starti}{\parindent0pt\it}  %% start an italic line
\newcommand{\startb}{\parindent0pt\bf}  %% start a boldface line
\newcommand{\mc}{\multicolumn}  %% multicolumn used in tables
\newcommand{\lrge}{\ge_{\rm lr}}        %%
\newcommand{\lrle}{\le_{\rm lr}}        %%

\newtheorem{thm}{Theorem}
\newtheorem{lemma}{Lemma}
\newtheorem{prop}{Proposition}
\newtheorem{coro}{Corollary}
\newtheorem{exm}{Example}
\newtheorem{exm*}{Example}
\newtheorem{rem}{Remark}

\newtheorem{asm}{Assumption}
\newtheorem{propty}{Property}
\newtheorem{proofdelayed}{Proof}[section]

\renewcommand{\theasm}{(\Alph{asm})}

\def\wp{{\rm w.p.}\quad}
\def\res{{\rm res}}
\def\dcum{D^{\rm cum}}
\def\dpp{{\partial\over{\partial {\bar f}_j}}}
\def\dppi{{\partial\over{\partial {\bar f}_i}}}
\def\dk{{\partial\over{\partial k_j}}}
\def\dpi{{{\partial\pi}\over{\partial {\bar f}_j}}}
\def\gradc{\nabla C}
\def\bfbar{{\bf {\bar f}}}
\def\sigtau{\sd (\tau_j)}
\def\Lin{L^{\rm in}}
\def\Lout{L^{\rm out}}
\def\lin{\ell^{\rm in}}
\def\lout{\ell^{\rm out}}
\def\eqd{{\buildrel {\rm d} \over =}}
\def\eqa{\, {\buildrel \sim \over = }\,}
\def\t{\tau}
\def\tg{{\tilde g}}

\def\ssk{\vskip .3cm}
\def\nd{\quad{\rm and}\quad}
\def\ind{\chi}

\newcommand{\bbP}{\mathbb P}
\newcommand{\bbQ}{\mathbb Q}
\newcommand{\bbQP}{{\mathbb Q}\times{\mathbb P}}
\newcommand{\lP}{L_\bbP}
\newcommand{\lQ}{L_\bbQ}
\newcommand{\lQP}{L_{\bbQ\times\bbP}}
\newcommand{\Sample}{\Sigma}
\newcommand{\Proj}{\Pi}
\newcommand{\Hamil}{\mathcal H}
\newcommand{\Hamiltonian}{\mathbf H}

\newcommand{\bbg}{\mathbf g}
\newcommand{\bbf}{\mathbf f}
\newcommand{\bbU}{\mathbf U}
\newcommand{\bbV}{\mathbf V}

\newcommand{\cP}{\mathcal{P}}
\newcommand{\cQ}{\mathcal{Q}}
\newcommand{\cT}{\mathcal{T}}
\newcommand{\cS}{\mathcal{S}}

\newcommand{\supp}{\rm supp\,}
\newcommand{\suppQ}{\mathfrak{Q}}
\newcommand{\suppP}{\mathfrak{P}}

\usepackage[figuresright]{rotating}

\begin{document}

\begin{frontmatter}



%% e.g. \dochead{17th International Conference on Dynamical Processes in Excited States of Solids}

\title{Polynomial convergence of iterations of certain random operators in Hilbert space}


\author{Soumyadip Ghosh}
\author{Yingdong Lu}
\author{Tomasz Nowicki}

%\address[SG]{IBM T.J. Watson Research Center, Yorktown Heights, NY 10598, USA}
%\address[YL]{IBM T.J. Watson Research Center, Yorktown Heights, NY 10598, USA}
%\address[TN]{IBM T.J. Watson Research Center, Yorktown Heights, NY 10598, USA}
\address{IBM T.J. Watson Research Center, Yorktown Heights, NY 10598, USA}

\begin{abstract}
We study the convergence of a random iterative sequence of a family of operators on infinite dimensional Hilbert spaces, inspired by the Stochastic Gradient Descent (SGD) algorithm in the case of the noiseless regression, as studied in~\cite{BerthialBachGaillard}. We identify conditions that are strictly broader than previously known for polynomial convergence rate in various norms, and characterize the roles the randomness plays in determining the best multiplicative constants. Additionally, we prove almost sure convergence of the sequence. 
\end{abstract}


\end{frontmatter}

% \linenumbers

\section{Introduction}\label{sec:intro}
%%%%%%%%%%%%%%%%%%%%%%%%%%%%%%%%%%%%%%%%%%%%
On a real Hilbert space $\bbH$ with inner product $\scpr{\cdot}{\cdot}$, define a family of rank 1 operators $\bbSx$ for $\xx\in\bbH$, and for given $\gamma\in [0,1)$ a family of operators $\bbTx$ acting on $\bbH$ by %and their averages $\bbS$ and $\bbT$ by
\begin{align}\label{def:Sx Tx}
    \bbSx:\bbH\ni\teta\mapsto\scpr{\teta}{\xx}\xx\in\bbH, \quad&\quad\bbTx:\bbH\ni\teta\mapsto\teta-\gamma\bbSx\teta\in\bbH\,.
    %\\
    %\label{def: S T}
    %  \bbS=\ex[\bbSx]:\bbH\to \bbH,\quad&\quad \bbT=\ex[\bbTx]:\bbH\to \bbH\,.
\end{align}
The operator $\bbTx$ is motivated by the steps of the 
\emph{stochastic gradient descent} (SGD) algorithm for a noiseless linear regression problem in infinite dimension, see e.g.~\cite{BerthialBachGaillard}.
Assume  that there exists an optimal parameter $\vartheta^*\in \bbH$ such that the data $y\in\Real$ and $\xx\in\bbH$ always satisfy $y=\scpr{\varteta^*}{\xx}$.
In SGD applications, the task of determining $\varteta^*$
with respect to the independent sampling $(\xx(1),y(1)),\dots,(\xx(n),y(n)),\dots$ 
%and
using the cost function ${\mathcal{L}}(\varteta|\xx)=(y-\scpr{\varteta}{\xx})^2=\scpr{\varteta-\varteta^*}{\xx}^2$ 
(derived from the assumption ) 
is carried by the following iterative scheme: given initial $\varteta_0\in \bbH$ (usually for practical reasons $\varteta_0=0$, but the convergence should not depend on it) we set
\begin{align}
\label{eqn:SGD_step}
\varteta(n+1)=\varteta(n)-\frac{\gamma}{2}\frac{\partial{\mathcal{L}}}{\partial\varteta}(\varteta(n))=
\varteta(n)-\gamma\scpr{\varteta(n)-\varteta^*}{\xx(n)}\cdot \xx(n)\,.
\end{align}
The parameter $\gamma>0$ is a small step size along the negative gradient of the cost function.   
%see Section \ref{sec:connection}. 
%~\cite{BerthialBachGaillard}. 
In Equation~\eqref{def:Sx Tx}, $\teta=\varteta-\varteta^*$ represents the difference between the output of the algorithm and the optimum. %\commentTN
{In terms of $\teta$ the cost function is equal to $\mathcal{L}(\teta|\xx)=\scpr{\teta}{\xx})^2$.}
%\commentTN
{To prove that $\varteta(n)\to\varteta^*$ is now equivalent to prove that  ${\teta}(n)\to 0$. We note that the properties of this model are invariant under a normalization, \emph{i.e.} a rescaling of the variables $\xx$ and $y$ by some (same) constant and the parameter $\gamma$ by the square of its inverse.
}
\medskip

Conceptually, $\bbSx$ projects $\teta$ to the $\xx$ direction (with the factor $\|\xx\|^2$), and $\bbTx$ takes a proportion $\gamma$ of the image of the projection away from the original $\teta$. When $\bbTx$ is iterated for randomly selected $\xx$, and for $\gamma$ small enough, one would expect that the image, hence the error of the algorithm, eventually vanishes. 

\commentOUT{
It is observed in~\cite{BerthialBachGaillard} that the convergence rate in a square norm for the random iteration sequence have polynomial lower and upper bounds. However, characterization of the bounds depends upon the regularity of both the initial state and the distribution of the random sequence. Hence it is not immediate to see that the gap between the lower and the upper bound can be closed readily, and it is deemed as an open problem.  
%{\color{blue}
With a different approach, 
}

 We prove %to conclude 
 the polynomial convergence rate of the average of the sequence which is explicitly determined only by the regularity of the initial state, Theorem \ref{thm:avg upper bound} and \ref{thm:avg lower bound}. For convergence of the second moment, under a condition  on the regularity of the random distribution ({\bf Assumption~\ref{asm:main}}), 
 %which is weaker than the ones in~\cite{BerthialBachGaillard}, 
 the convergence rate remains the same, Theorem \ref{thm:upper bound}. In another words, under~{\bf \ref{asm:main}}, the regularity of the random sequence only affects the coefficient not the order of the polynomial convergence. 
% In Section \ref{sec:prop bounds}, Remark \ref{rem:two} also provides a class of natural examples when~{\bf\ref{asm:main}} is violated, consequently the regularity of the iterations is lost and convergence rate is denied. 
 Additionally  we demonstrate almost sure convergence of the sequence, Theorem \ref{thm:asc}. 
%}

The rest of the paper will be organized as follow: in Section \ref{sec:results}, we  present our main results and their implications; in Section \ref{sec:properties of S}, we discuss the basic properties of the key operators and some key assumptions of the papers; the proofs the convergence rates are presented in Section \ref{sec:prop bounds}, while the proof of the almost sure convergence is presented in Section \ref{sec:asc}, with proofs of technical lemmata collected in Section \ref{sec:technical lemmata}. 

\commentOUT{Finally we compare the our results with the results of the paper \cite{BerthialBachGaillard} which inspired this article.} 

\commentOUT{
Following~\cite{BerthialBachGaillard} we consider a (noiseless) model 
$y=\scpr{\varteta^*}{\xx}$, where $y\in\Real$ and both  $\varteta, \xx\in \bbH$ are elements of the same Hilbert space with the scalar product $\scpr{\cdot}{\cdot}$. The parameter $\varteta^*$ exists but is unknown, our goal is to prove that a Stochastic Gradient Descent (SGD) algorithm, using i.i.d. samples of (features vector) $\xx$ from some fixed but unknown distribution provides convergent approximations.
We assume 
that the sequences of feature sample vectors $(\xx(0),\xx(1),\dots,\xx(n),\dots)$ produce the corresponding sequence of numbers $(y(0),y(1),\dots,y(n),\dots)$ which equal exactly to  $y(i)=\scpr{\varteta^*}{\xx(i)}$. 
}
%%%%%%%%%%%%%%%%%%%%%%%%%%%%%%%%%%%%%%%%%%%%
\section{Main results} \label{sec:results}
%%%%%%%%%%%%%%%%%%%%%%%%%%%%%%%%%%%%%%%%%%%%
For \commentTN{iid}\commentOUT{independent and identically distributed} random variables $\xx(1),\dots,\xx(n),\dots\in \bbH$ and given $\teta(0)$, the recursive definition~\eqref{eqn:SGD_step} becomes
\begin{align}
\label{eqn:update}
\teta(n+1)={\teta}(n)-\gamma\scpr{\teta\commentTN{(n)}}{\xx(n)}\cdot \xx(n)=\bbT_{\xx(n)}(\teta(n))\,.
\end{align}
Furthermore, define the average operators $\bbS$ and $\bbT$ of $\bbSx$ and $\bbTx$ by
\begin{align}
%\label{def:Sx Tx}
%    \bbSx:\bbH\ni\teta\mapsto\scpr{\teta}{\xx}\xx\in\bbH, \quad&\quad\bbTx:\bbH\ni\teta\mapsto\teta-\gamma\bbSx\teta\in\bbH\,,\\
    \label{def: S T}
      \bbS=\ex[\bbSx]:\bbH\to \bbH,\quad&\quad \bbT=\ex[\bbTx]:\bbH\to \bbH\,,
\end{align}
where the symbol $\ex[\cdot]$ denotes the expected value w.r.t. the distribution of the vector $\xx$, but also the expected value w.r.t. the product distribution of the samples. We assume that $\bbS$ and $\bbT$ are bounded and well defined on $\bbH$, for which it is enough to assume that $\ex[\|\xx\|^2]<\infty$.  We note that $\bbS$ (as we shall see being symmetric), when defined on all $\bbH$,  is bounded by 
%\href{https://en.wikipedia.org/wiki/Hellinger%E2%80%93Toeplitz_theorem
Hellinger–Toeplitz Theorem, (for basic materials and theorems of functional analysis used in this paper, see, e.g.~\cite{reed1981functional})
even without the condition on $\ex[\|\xx\|^2]$. 
The operators have finite norms, in particular $\|\bbSx\|^2=\|\xx\|^2<\infty$.
%We can shift the variable $\varteta$ to $\teta\mapsfrom\varteta-\varteta^*$.
%With the variable $\teta$ the cost function at the $n$-th SGD step (depending on the sample $\xx(n)$ at this step)  and the updates are expressed as follows,
%\begin{align}\label{eqn:deftilde theta}
% &\mathcal{L}(\teta(n)|\xx(n))=\scpr{\teta(n)}{\xx(n)}^2\qquad\text{and}\nonumber\\
%&\teta(n+1)={\teta}(n)-\gamma\scpr{\teta}{\xx(n)}\cdot \xx(n)=\bbT_{\xx(n)}(\teta(n))\,,
%\end{align}
Because $\bbS$ is also non-negative, the powers $\bbS^\beta$ are well defined for (some) real values of $\beta$, certainly for all $\beta\ge 0$, $\bbS^0={\rm Id}$ and $\bbS^1=\bbS$. 
\begin{exm}
\label{example1}
The basic example illustrating the variable $\xx$ to keep in mind is related to the %\href{https://en.wikipedia.org/wiki/Gaussian_free_field}{
Gaussian Free Field~\cite{GaussianFreeFieldSheffield}.  Let $(\be_i)_{i=1}^\infty$ be any orthonormal basis in $\bbH$. Define the random variable $\xx= \sum_{i=0}^\infty x_i\be_i$, where $x_i$ are independent variables with mean 0 and variances $\ex [x_i^2]=\lambda_i$, note that for $i\not=j$, $\ex[x_ix_j]=\ex[x_i]\ex[x_j]=0$. In this setting, with $\teta=\sum_i\theta_i\be_i$ we have $\scpr{\teta}{\xx}=\scpr{\sum_{i=0}^\infty\theta_i \be_i}{\sum_{j=0}^\infty x_j \be_j}=\sum_i(\theta_i x_i)$ and
\begin{align*}
\bbS \teta
&=
\ex[\bbSx \teta]=\ex[\scpr{\teta}{\xx}\xx]
=\ex\left[\sum_i(\theta_i x_i)\cdot\sum_k (x_k\be_k)\right]
%\\ &
=\sum_k\sum_i\left(\theta_i\ex[x_i x_k]\be_k\right)
=
\sum_{i}\theta_i\ex[x_i^2]\be_i=\sum_{i=0}^\infty \lambda_i\theta_i \be_i\,.
\end{align*}
We conclude that   $\bbS\,\teta\in\bbH$ for every $\teta\in\bbH$\quad iff \quad $\lambda_j$ are uniformly bounded.
\end{exm}

We shall investigate the rate of convergence by using the "norms"
\begin{align*}%\label{eqn:phi}
\varphi_\beta&:\bbH\to \Real,  \quad\bbH\ni \teta\mapsto\varphi_\beta(\teta)=\scpr{\teta}{\bbS^{-\beta}\teta}=:\|\teta\|_\beta^2\,,\\
\text{given}\quad \teta(0),\quad \phi_n&:\Real\to \Real,  \quad\Real\ni \beta\mapsto\phi_n(\beta)=\ex[\varphi_\beta (\teta(n))]=\ex[\|\teta(n)\|_\beta^2]\,.
\end{align*}

The numbers $\phi$ depend on the starting $\teta(0)$ but, due to the expected value, not on the choice of the samples $(\xx)$.
We introduce the limits of applicable $\beta$, for $\teta, \xx\in\bbH$, as,
\begin{equation}\label{eqn:def alpha teta xx}
    \alpha(\teta)=\sup\{\beta:\varphi_\beta(\teta)<\infty\}\qquad \text{and}\qquad
\aalpha=\sup\{\beta:\ex[\varphi_\beta(\xx)]<\infty\}\,. 
\end{equation}
%For $\teta\in\bbH$ w
We have $\alpha(\teta)\ge 0$ and, as we shall see (Lemma \ref{lem:second moment of xx}), $\aalpha\le 1$.

%\begin{lemma}[The upper bound on eigenvalues]\label{lem:upper buond on lambda}
%\item

%\subsection{Connection to SGD and~\cite{BerthialBachGaillard}}
%\label{sec:connection}
\commentOUT{When in \eqref{eqn:SGD_step} we shift the variable $\varteta$ to $\teta\mapsfrom\varteta-\varteta^*$, then the cost function at the $n$-th SGD step (depending on the sample $\xx(n)$ at this step)  and the updates take the form presented above in Equation \eqref{eqn:update}.
}

\medskip
\commentOUT{
%\subsection*{The results of~\cite{BerthialBachGaillard}} 
In the article~\cite{BerthialBachGaillard} under 
\begin{asm}\label{asm:bounded xx}
$\varphi_\beta(\xx)$ is uniformly bounded for all $\xx\in\bbH$ (including $\varteta^*$)\,;
\end{asm}
the authors state that
\begin{enumerate}
\item 
if there is some $\underline{\al}$ such that both regularity properties of the target and data are satisfies (i.e. both $\varphi_{\underline \al}(\varteta^*)$ and $\ex[\varphi_{\underline \al}(\xx)]$ are finite), then there are constants $C_0,C_{-1}$ such that for all $n$:\quad 
$\phi_n(0)\le C_0n^{-\underline{\al}}$\quad and\quad $\phi_n(-1)\le C_{-1}n^{-(\underline{\al}+1)}$;
\item if there exists an $\overline{\al}$ such that  $\varphi_{\overline{\al}}(\varteta^*)$ or 
$\ex[\varphi_{\overline{\al}}(\xx)]$
are infinite then no such constants can be found. 
\item 
Because $\underline{\al}$ and $\overline{\al}$ are determined by the regularity of both target and data, it is considered an open problem for closing the gap between them. 
\end{enumerate}
It is also shown that {\bf Assumption~\ref{asm:bounded xx}} can be replaced by a more general one (Remark 3. in~\cite{BerthialBachGaillard}):
\begin{asm}\label{asm:Remark3}
$\exists_{\alpha>0}\,\forall_{\beta<\alpha}\,\exists_{R_\beta}\,\forall_{\teta\in\bbH}$ \quad
$
\ex[\scpr{\teta}{\xx}^2\varphi_{\beta}(\xx)]\le R_{\beta}\ex[\scpr{\teta}{\xx}^2]=R_{\beta}\varphi_{-1}(\teta)
$.
\end{asm}
%\subsection{Statements} 
%\subsection*{Our results}\ \\
Our approach and results are different.  
}

First two Theorems bound $\phi_n(\beta)$ for averages $\ex[\teta(n)]=\bbT^n\teta(0)$, depending only on 
 $\varphi_\beta(\teta(0))$, where 
 $\teta(0)=\varteta(0)-\varteta^*\,(=-\varteta^*)$.
 \begin{thm}[Upper bound for the average $\teta(n)$]\label{thm:avg upper bound}
 Given $\teta(0)=\teta$ and $\bbT^n\teta=\ex[\teta(n)]$ we have,
 \begin{align*}%\label{res:upper}
\text{for every}\quad n,\quad\|\bbT^n\teta\|^2\le \exp(-\beta)\left(\frac{\beta}{n}\right)^\beta\cdot\|\teta\|_\beta^2\,.
\end{align*}
\end{thm}
%\begin{proof}
%See Proposition \ref{prop:avg upper bounds} with $\kappa=0$.
%\end{proof}
%
\begin{thm}[Lower bound for the average $\teta(n)$]\label{thm:avg lower bound}
Given $\teta(0)=\teta$ and $\bbT^n\teta=\ex[\teta(n)]$,  
for any sequence $(t_n)>0$ such that $\sum_n 1/(nt_n)<\infty$, we have,
\begin{align*}%\label{res:lower}
\text{if}\quad \|\bbT^n\teta\|^2\le \frac{1}{n^\beta t_n}\quad\text{for every}\quad n, 
\quad\text{then}\quad \|\teta\|_\beta^2<\infty\,.
\end{align*}
\end{thm}
\emph{Examples of slow increasing sequences $t_n$  with $\sum_n 1/(n t_n)<\infty$ are $n^\epsilon$, $(\ln n)^{1+\epsilon}$ or $\ln n\cdot(\ln\ln n)^{1+\epsilon}$ etc. with any $\epsilon>0$.} 
%\begin{proof}
%See Proposition \ref{prop:avg lower bounds} with $\kappa=0$.
%\end{proof}
%

\medskip

%\commentTN
{
 In order to produce the upper bound of the square of $\teta(n)$ 
 \commentOUT{in~\cite{BerthialBachGaillard}}
 we need an additional assumption. First let's define a family of inequalities:
 \begin{align}\label{eqn:c kappa beta}
 \ex\left[\scpr{\teta}{\xx}^2\scpr{\xx}{\bbS^{-\beta}\xx}\right]&\le C_{\beta\kappa}\scpr{\teta}{\bbS^{1-\kappa}\teta}
     \tag{$C_{\beta\kappa}$}\quad(\,=\,C_{\beta\kappa}\,\varphi_{\kappa-1}(\teta)\,)\,.
 \end{align}
  \commentOUT{weaker than {\bf Assumption~\ref{asm:Remark3}}  }
\begin{asm}
\label{asm:main}
There is an $\alpha>0$ such that the distribution of $\xx$ satisfies~\eqref{eqn:c kappa beta} with $\kappa=\beta<\alpha$. 
For every  $\beta<\alpha$ there is a $C_\beta\,(:=C_{\beta\beta})$  such that for every $\teta\in\bbH$:
\[
\ex\left[\scpr{\teta}{\xx}^2\varphi_\beta(\xx)\right]\le C_{\beta}\scpr{\teta}{\bbS^{1-\beta}\teta}=
C_{\beta}\varphi_{\beta-1}(\teta)\,.
\]
\end{asm}
}

\begin{thm}[Upper bound for the average  $\|\teta(n)\|^2$]\label{thm:upper bound}
Assuming~{\bf \ref{asm:main}}, for any $0\le\beta<\alpha(\teta)$ if we take $\gamma<2/C_\beta$ then \quad
$
\ex[\|\teta(n)\|^2]\le \calO(1)n^{-\beta}
$.
\end{thm}
\noindent
The average of $\|\teta(n)\|^2$ is lower bounded by 
%$\ex[\|\teta(n)\|^2]=\phi_n(0)$, converges to 0 not faster than 
$\|\ex[\teta(n)]\|^2$.  Theorem~\ref{thm:avg lower bound}\, applied to $\beta>\alpha(\teta)$ allows us to take $t_n=C n^{\beta-\kappa}$, $\alpha(\teta)<\kappa<\beta$. 
Thus $\|\bbT^n\teta\|^2$ cannot be bound by $C n^{-\beta}=C n^{-\kappa} t_n$,  as it  would imply $\|\teta\|_\kappa^2<\infty$ a contradiction to $\kappa>\alpha(\teta)$. %This provides a 

%\begin{proof}
%See Proposition \ref{prop:upper bound rnd} with $\kappa=0$.
%\end{proof}
\medskip
Convergence in norms, including the average, implies that $\teta_n$ converge to zero in probability as a sequence of random variables in $\bbH$, see, e.g.~\cite{ledoux1991probability}. It is natural to examine the almost sure convergence of the $\teta_n$. 

\begin{thm}[Almost sure convergence]
\label{thm:asc}
If $\ex[\|\xx\|^4] < \infty$ and $\delta:=\inf_{||z||=1}\ex[\scpr{z}{\xx}^2]>0$ 
then the sequence $\teta_n$ converges to zero almost surely
 for $\gamma<\delta/\ex[\|\xx\|^4] $.
If for some $M>0$, $\|\xx\|^2\le M$ almost surely  then such convergence occurs for $\gamma\le 2/M$.
%ad or
\end{thm}
\noindent
We shall see in Proposition~\ref{prop:beta kappa} that the condition in Theorem~\ref{thm:asc},  $\ex[\|x\|^4]<\infty$,  is satisfied under {\bf Assumption~\ref{asm:main}}
\commentOUT{, so also under~{\bf \ref{asm:Remark3}}, while the boundedness of $\|\xx\|^2\le M$ corresponds to ~{\bf\ref{asm:bounded xx}} with $\beta=0$}.

\medskip
For proofs of Theorems \ref{thm:avg upper bound}, \ref{thm:avg lower bound}, and \ref{thm:upper bound}, see Section~\ref{sec:prop bounds}, Propositions \ref{prop:avg upper bounds}, \ref{prop:avg lower bounds}, and \ref{prop:upper bound rnd} with $\kappa=0$. For proof of Theorem~\ref{thm:asc} see Section~\ref{sec:asc}.

%\commentTN
{
\begin{exm}\label{example2}
Let the distribution of $x$ in Example \ref{example1} 
%be such that in the ON basis of eigenvectors of $\bbS$ the coordinates $x_i$ are independent (a weaker condition can be used) 
be even with the property that 
$y_i=x_i^2$ has a ($\Gamma$-)density $t^{\lambda_i-1}\exp({-t}) /\Gamma(\lambda_i)$ for $t\in [0, \infty)$. Therefore, 
$\ex[x_i^2]=\ex[y_i]=\int_0^\infty t^{\lambda_i}\exp(-t)\,dt/\Gamma(\lambda_i)=\Gamma(\lambda_i+1)/\Gamma(\lambda_i)=\lambda_i$.
Similarly $\ex[x_i^4]=\ex[y_i^2]=\lambda_i(1+\lambda_i)$.
Given $\alpha\in(0,1)$ let $0<\lambda_i\searrow 0$ be such that 
$\sum_{i=0}^\infty \lambda_i^{1-\beta}=:K_\beta<\infty$ for any $\beta\in(0,\alpha]$. 
\\
Some elementary calculation (see Section~\ref{sec6:example2}) using which using independence, 0 mean and calculated moments provide that  LHS of~\ref{asm:main} is equal to
\begin{align*}
 \ex\left[\scpr{\teta}{\xx}^2\varphi_\beta(\xx)\right]=\ex\left[\left(\sum_i\theta_i x_i\right)\cdot\left(\sum_k \theta_i x_k\right)\cdot\left(\sum_j \lambda_j^{-\beta}x_j^2\right)\right]
\sum_{i}\theta_i^2\lambda_i\cdot K_\beta
+\sum_{i}\theta_i^2\lambda_i^{1-\beta}\,.
\end{align*}
Now  we can use Assumption \ref{asm:bounds on lambda} below (which is not essential, as without it we would just have a less pleasant constant) and get 
an upper bound for the  LHS by $(K_\beta+1)\scpr{\teta}{\bbS^{1-\beta}\teta}$, which is the RHS of~\ref{asm:main} with $C_\beta=K_\beta+1$. That proves that our example satisfies Assumption~\ref{asm:main}.
\end{exm}
\begin{rem}
This example also shows that the bound in \ref{asm:main} is accurate. The  collection of $\teta$'s satisfying the inequality~\eqref{eqn:c kappa beta} with $\kappa=\beta$, that is Assumption~\ref{asm:main}, is larger than the collection satisfying a stronger  assumption, 
the inequality~\eqref{eqn:c kappa beta} with $\kappa<\beta$. In particular this applies to $\kappa=0$, which is the condition used  in~\cite{BerthialBachGaillard}. Example \ref{example2} provides a family of distributions that satisfy Assumption~\ref{asm:main} but not \eqref{eqn:c kappa beta} with $\kappa=0$.
Indeed, as $\teta$ is arbitrary we can take in our example $\teta=\bbe_i$. Then the LHS will be equal (as in the last expression above) to  $K_\beta\lambda_i +\lambda_i^{1-\beta}$ which,  for any given $C_{\beta\kappa}$ and for sufficiently large $i$ (and therefore small $\lambda_i$) is larger than the RHS  equal to $C_{\beta\kappa}\lambda_i^{1-\kappa}$ due to $1-\kappa>1-\beta$ and $\lambda_i\searrow 0$.
\end{rem}
}
%%%%%%%%%%%%%%%%%%%%%%%%%%%%%%%%%%%%%%%%%%%%
\section{Properties of the operators}\label{sec:properties of S}
%%%%%%%%%%%%%%%%%%%%%%%%%%%%%%%%%%%%%%%%%%%%
In this section we present basic properties of the operators, which can be easily deduced directly from the definitions.
\commentOUT{In Introduction,  we defined a family of linear operators (of rank 1)}
%\commentTN
{Recall the definitions of $\bbSx$ and $\bbTx$ in~\eqref{def:Sx Tx} and  their averages $\bbS$ and $\bbT$ in~\eqref{def: S T}. We assume that both $\bbSx$ and $\bbS$ are bounded and well defined for all $\teta\in\bbH$.} 

 %\subsection*{Properties of $\bbSx$ and $\bbS$}
 \begin{propty}[$\bbSx$ and the average $\bbS$ are symmetric and non-negative]\label{lem:symm and nonneg}\ 
 \begin{itemize}
     \item $\scpr{\eeta}{\bbSx\teta}=\scpr{\eeta}{\xx}\scpr{\teta}{\xx}=\scpr{\teta}{\bbSx\eeta}$;
     \item
    $\scpr{\eeta}{\bbS\teta}=\scpr{\eeta}{\ex[\bbSx\teta]}=\ex[\scpr{\eeta}{\bbSx\teta}]=
    \ex[\scpr{\teta}{\bbSx\eeta}]=\scpr{\teta}{\bbS\eeta}$;
    \item non-negativity:
    $\scpr{\teta}{\bbSx\teta}=\scpr{\teta}{\xx}^2$.
 \end{itemize}
% and so is their average $\bbS=\ex[\bbSx]$.
\end{propty}

 \begin{propty}[$\bbS$ admits an orthonormal (ON) basis of eigen-vectors]\label{lem:ON basis}\ 
 \begin{itemize}
    \item As the operator $\bbS$  is symmetric, non-negative and defined on all $\bbH$, it has an ON basis $(\bbe_i)$ of eigen-vectors, with corresponding bounded non-negative eigenvalues $(\lambda_i)$.
    %, see, e.g.~\cite{hille1996functional}.
    \item
    If in this basis $\teta=\sum\theta_i\bbe_i$ then $\bbS\teta=\sum\lambda_i\theta_i\bbe_i$. 
\end{itemize}
\end{propty}
\begin{propty}[The moments of $\xx$]\label{lem:second moment of xx}\ 
\begin{itemize}
\item 
Each feature coordinate $x_i$ of $\xx$ in the ON basis $(\bbe)$ has finite second moment:
$\ex[x_i^2]=\lambda_i$.\\
Using $\scpr{\bbe_i}{\bbe_i}=1$ and $\scpr{\bbe_i}{\xx}=x_i$ for the features vector $\xx=\sum_i\bbe_i$ we obtain,
\[
\lambda_i=\scpr{\bbe_i}{\lambda_i \bbe_i}=\scpr{\bbe_i}{\bbS \bbe_i}=\scpr{\bbe_i}{\ex[\bbSx\bbe_i]}=
\ex[\scpr{\bbe_i}{x_i \xx}]=\ex[x_i^2]\,.
\]
\item
The coordinates of $\xx$ in the ON basis $(\bbe)$ are de-correlated: $\ex[x_ix_j]=0$ (\emph{are uncorrelated if $\ex[\xx]=0$}). \\
Using $\bbe_i+\bbe_j$ and the orthonormality we get 
\begin{align*}
\lambda_i+\lambda_j&=\scpr{\bbe_i+\bbe_j}{\lambda_i \bbe_i+\lambda_j\bbe_j)}=
\scpr{\bbe_i+\bbe_j}{\bbS (\bbe_i+\bbe_j)}
\\
&=\scpr{\bbe_i+\bbe_j}{\ex[\bbSx(\bbe_i+\bbe_j)]}
=\ex[\scpr{\bbe_i+\bbe_j}{\bbSx(\bbe_i+\bbe_j)}]
\\
&=\ex[\scpr{\bbe_i+\bbe_j}{\xx}^2]=\ex[(x_i+x_j)^2]
=\lambda_i+2\ex[x_ix_j]+\lambda_j
\end{align*}

\item Special form of $\bbS$ in the ON basis.\\
$
\ex[\scpr{\teta}{\bbSx\teta}]=\ex[\scpr{\teta}{\xx}^2]=\ex[(\sum\theta_i x_i)^2]=\ex[\sum(\theta_ix_i)^2]=\sum(\theta_i^2\ex[x_i^2])=\sum\theta_i^2\lambda_i=\scpr{\teta}{\bbS \teta}
$.
\end{itemize}
\end{propty}
\emph{
    We note that when $\lambda_i=0$ we have $\ex[x_i^2]=0$, so that $x_i=0$ \emph{a.s.} and we may restrict ourselves to the closure of the subspace $\{\boldsymbol{h}=\sum_{\lambda_i>0} h_i\bbe_i\}\subset \bbH$, where $\bbS \teta=0$ only when $\teta=0$. 
    }
    
We have $\ex[\varphi_\beta(\xx)]=\ex[\scpr{\xx}{\bbS^{-\beta}\xx}]=\ex[\scpr{\sum_i x_i\bbe_i}{\sum_j x_j\lambda_j^{-\beta}\bbe_j}]=
 \ex[\sum_i \lambda^{-\beta}x_i^2]=\sum_i\lambda_i^{{-\beta}}\ex[x_i^2]=\sum_i\lambda_i^{1-\beta}$. In particular 
 the sum is infinite for $\beta\ge 1$ as $\lambda_i$ are bounded so  $\aalpha\le 1$. 
 Also $\ex[\|\xx\|^2]=\ex[\sum x_i^2]=\sum \lambda_i$. 
From now on we shall use 
\begin{asm}\label{asm:bounds on lambda}
For any eigenvalue  $\lambda$ in the spectrum of $\bbS$ we have 
$    %\label{ass:lambda 0 1/2}
    0<\lambda<\frac{1}{2}<1$.
\end{asm}
This is not a loss of generality. The operator is continuous, hence bounded and its spectrum is compact. It is positive and symmetric. Let $\lambda_0=\sup\lambda$. As we are interested in the iterations of $\bbTx={\rm I}-\gamma\bbSx$ for small $\gamma$ we may assume that $\gamma<1/2\lambda_0$ by changing either $\xx$ (and $y$) to $\xx/{2\lambda_0}$ (and to $y/{2\lambda_0}$) 
or changing $\bbSx$ to $\teta\mapsto\scpr{\teta}{\xx}\cdot \xx/2\lambda_0$, effectively using $\gamma'=\gamma\cdot 2\lambda_0$.
%\end{propty}

Using the ON basis the operators $\bbS^\kappa:\bbH\to \bbH$ can be now defined by
$\bbS^{\kappa}\teta=\sum\lambda_i^\kappa\theta_i\bbe_i$.

With the definitions~\eqref{eqn:def alpha teta xx} of $\alpha(\teta)$ and $\alpha$ from Section~\ref{sec:intro} we have,
\begin{propty}[Bounds on the powers $\bbS^{-\beta}$]\label{lem:def of alpha}
(1) Given $\teta$, $\|\teta\|^2_\beta$ is an increasing function of $\beta$; (2) $\alpha(\teta)\ge 0$; and (3) $\aalpha\le 1$, independently of the distribution of data $\xx$. If $\ex[\|\xx\|^2]<\infty$ then $\aalpha\ge 0$ and $\sum\lambda_i<\infty$. 
\end{propty}
\commentOUT{As $\|\teta\|_0^2=\|\teta\|^2<\infty$ for $\teta\in \bbH$ we have $\alpha(\teta)\ge 0$.}

\commentOUT{
\begin{rem}
Lemma~\ref{lem:def of alpha} asserts that $\scpr{\teta}{\bbS^{1-\beta}\teta}=\|\teta\|_{\beta-1}^2$  is an increasing function of $\beta$. 
Given $\kappa$ the LHS of the inequality~\eqref{eqn:c kappa beta} remains the same, hence the minimal constant $C_{\beta\kappa}$ in its RHS is decreasing with increased $\beta$. Similarly given $\beta$ the minimal  $C_{\beta\kappa}$ must increase with $\kappa$. 
In particular given $\kappa$, if the distribution of $\xx$ satisfies  \eqref{eqn:c kappa beta} for some $\beta$ it also satisfies it for any larger $\beta'$ with a smaller constant.
\end{rem}
}
\commentOUT{{\bf \ref{asm:Remark3}} implies~{\bf\ref{asm:main}}. }

\begin{prop}[Bounds on moments]
\label{prop:beta kappa}
Suppose that Assumption~\ref{asm:main} is satisfied with $\beta=0$, which means that there exists a $C_0$ such that for all $\teta$ we have $\ex[\scpr{\teta}{\xx}^2\|\xx\|^2]\le C_0\scpr{\teta}{\bbS\,\teta}$, then
\[
\ex[\|\xx\|^2]^2\le\ex[\|\xx\|^4]\le C_{0}\ex[\|\xx\|^2]\le C_{0}^2\,.
\]
\end{prop}
For proof see Section~\ref{sec6:prop:beta kappa}. 
Proposition~\ref{prop:beta kappa} implies that both the second and the fourth moments of $\xx$ are finite. 

%\commentTN{
%\begin{rem}[A consequence of~\ref{asm:Remark3} from~\cite{BerthialBachGaillard}] 
%If $\ex[\scpr{\teta}{\xx}^2\varphi_{\beta}(\xx)]\le R_{\beta}\scpr{\teta}{\xx}^2=R_{\beta}\varphi_{-1}(\teta)$ 
%then using $\teta=\bbe_i$ we get $\ex[x_i^2\varphi_{\beta}(\xx)]\le R_{\beta}\ex[x_i^2]=R_{\beta}\lambda_i$ and summing up after multiplying by $\lambda_i^{-\kappa}$ we get
%$
%\ex[\varphi_{\kappa}(\xx)\varphi_{\beta}(\xx)]\le R_\beta \ex[\varphi_{\kappa}(\xx)]$. In particular for $\kappa=\beta$ we have 
%$\ex[\varphi_{\beta}(\xx)]^2\le\ex[\varphi^2_{\beta}(\xx)]\le R_\beta %\ex[\varphi_{\beta}(\xx)]\le R_\beta^2$. This means that under~\ref{asm:Remark3} $\aalpha$ can %be defined also as $\sup\{\beta:\ex[\varphi^2_\beta(\xx)]<\infty\}$.
%\end{rem}
%}

%\subsection*{The operator $\bbT$}
%The SGD algorithm can be expressed as the iterations of the operator $\bbTx$ where $\xx$ is chosen in an i.i.d. fashion at every iterate.
Given $\teta(0)$ and a sample sequence $(\xx(i))$ we have  $\teta(n+1)=\bbT_{\xx(n)}\teta(n)$ and their averages $\ex[\teta(n+1)]=\ex[\bbT_{\xx(n)}\teta(n)]=\bbT\ex[\teta(n)]$, by linearity and independence. Indeed, the random variable $\teta(n)$ does not depend on the last element of the sample sequence, while the operator $\bbSx$ depends exclusively on it. 

\begin{propty}[Evolution of averages]
\label{lem:evolution of averages}
The evolution of averages  follows the deterministic dynamics of  $\bbT$.
\[
\ex[\teta(n+1)]=\bbT\ex[\teta(n)]\,.
\]
\end{propty}

%The average value follows the iterations of $\bbT$:
%\begin{align*}
%\ex[\teta(n+1)]&=\ex[\bbTx\teta(n)]=\ex[\teta(n)]-\gamma\ex[\bbSx\teta {(n)}]=
%\ex[\teta(n)]-\gamma\bbS(\ex[\teta(n)])
%=\bbT(\ex[\teta(n)])\,.
%\end{align*}

In the ON basis, if  $\teta=\sum\theta_i\bbe_i$ then   $\bbT\teta=\sum_i(1-\gamma\lambda_i)\theta_i\bbe_i$, and its iterates are  
$\bbT^n\teta=\sum_i(1-\gamma\lambda_i)^n\theta_i\bbe_i$.  

\medskip
\noindent
\emph{If all $\lambda_i$'s are uniformly separated from 0, setting $\gamma<\min(1,1/\inf(\lambda_i))$ the iterates of the averages converge uniformly exponentially to 0, with the rate $1-\gamma\inf(\lambda_i)<1$. 
{If additionally the feature vector itself has a finite second moment then
$\sum\lambda_i=\ex[\scpr{\xx}{\xx}]<\infty$ and $\lambda_i\searrow 0$.
We may then assume that $(\lambda_i)$'s form a non-increasing sequence.} 
}

%%%%%%%%%%%%%%%%%%%%%%%%%%%%%%%%%%%%%%%
\section{Bounds on convergence}\label{sec:prop bounds}
%%%%%%%%%%%%%%%%%%%%%%%%%%%%%%%%%%%%%%%%%%%%%

Define a real function $  f(\lambda):=|1- \lambda|^{m}\lambda^{\beta}$.
 \begin{lemma}\label{lem:f(lambda)}
For any $m>0$ and $\tau>0$ there is a unique local maximum of $f$ at $\lambda_* =\frac{\tau}{m+\tau}\in(0,1)$ where 
we have
\[
   \exp\left(-\tau\frac{e}{e-1}\right)\left(\frac{\tau}{m}\right)^\tau\le f(\lambda_*)\le \exp(-\tau)\left(\frac{\tau}{m}\right)^\tau\,.
\] 
 Moreover for any $0<\epsilon\le 2$  there exists an $m>0$  such that the upper inequality holds also for $0\le\lambda \le 2-\epsilon$. (Proof: see Section~\ref{sec6:lem:f(lambda)}.)

\end{lemma}
As $m>0$ we can  write $|1-\lambda|^m=((1-\lambda)^2)^{m/2}$. We observe that $f(0)=f(1)=0$, $f(2)=2^\tau>1$.

\begin{prop}[Upper bound]\label{prop:avg upper bounds}
For any $\kappa<\beta$ we have $\|\bbT^n\teta\|_\kappa^2\le \|\teta\|^2_\kappa$ and
\begin{align*}%\label{eqn:prop:upper bound}
\phi_n(\kappa)=\|\bbT^n\teta\|_\kappa^2\le \exp(-\beta+\kappa)\left(\frac{\beta-\kappa}{2n\,\gamma }\right)^{\beta-\kappa}\cdot\|\teta\|_\beta^2\,.
\end{align*}
\end{prop}

\begin{proof}%%%%%%%%%%%%%%%%%Of the Proposition on averages
In the ON basis we have $\bbT^n\teta=\sum_i(1-\gamma\lambda_i)^n\theta_i\bbe_i$ and $\|\bbT^n\teta\|_\kappa^2=\sum_i\lambda_i^{-\kappa}(1-\gamma\lambda_i)^{2n}\theta_i^2
\le \sum_i\lambda_i^{-\kappa}\theta_i^2=\|\teta\|_\kappa^2$, by~{\bf\ref{asm:bounds on lambda}}.
Setting $\mu_i=\gamma\lambda_i$ 
%and using Lemma~\ref{lem:f(lambda)}  
we have, 
\begin{align*}%\label{eqn:lambda up}
\gamma^{\beta-\kappa}\|\bbT^n\teta\|_\kappa^2
&=\sum_i{\mu_i^{-\kappa}(1-\mu_i})^{2n}\gamma^{\beta}\lambda_i^\beta\lambda_i^{-\beta}\theta_i^2
=\sum_i\left({\mu_i^{\beta-\kappa}(1-\mu_i})^{2n}\right)\lambda_i^{-\beta}\theta_i^2\\
(\text{\small by Lemma~\ref{lem:f(lambda)} with $\tau=\beta-\kappa$})&\le \exp(\kappa-\beta)\left(\frac{\beta-\kappa}{2n}\right)^{\beta-\kappa}\sum_i\lambda_i^{-\beta}\theta_i^2=\exp(\kappa-\beta)\left(\frac{\beta-\kappa}{2n}\right)^{\beta-\kappa}\|\teta\|_\beta^2\,.
\end{align*}
\end{proof}

\begin{lemma}[Series and function $\Gamma$]\label{lem:series and Gamma}
For any $\alpha>0$ there exists a constant $K>0$ such that for every  $0<\mu<1/2$ and $0<\kappa<\alpha$ we have \quad
$\displaystyle
K\,\Gamma(\kappa)<\sum_n(1-\mu)^n(n\,\mu)^{\kappa}/n\le K^{-1}\Gamma(\kappa)
$,\quad
where, for $\mathfrak{Re}(z)>0$, $\Gamma(z)=\int_0^\infty e^{-t} t^{z-1}\,dt$.
  (Proof: see Section~\ref{sec6:lem:series and Gamma}.)

\end{lemma}

\begin{prop}[Lower bound]\label{prop:avg lower bounds}
Let the sequence $(t_n)>0$ be such that $\sum_n 1/(nt_n)<\infty$.
\begin{align*}
\text{if for some\ \ }0\le \kappa<\beta,\quad \phi_n(\kappa)=\|\bbT^n\teta\|_\kappa^2\le \frac{1}{n^{\beta-\kappa} t_n}
\text{for all\ \ }n, 
\quad{then}\quad \|\teta\|_\beta^2<\infty\,.
\end{align*}
\end{prop}

The arbitrary sequence $t_n$ in Proposition~\ref{prop:avg lower bounds} is mostly interesting in case $\|\teta\|_{\alpha(\teta)}=\infty$. 
 
\begin{proof}
We use again the convention $q_i=-\ln(1-\gamma\lambda_i)\in(0,\ln 4)$
\begin{align*}
    \|\bbT^n\teta\|_\kappa^2&=
\calO(1)\gamma^{\kappa-\beta}\sum_i\exp(-nq_i)(q_i)^{\beta-\kappa} \lambda_i^{-\beta}\theta_i^2\,,
\\
\infty&>\sum_n\frac{1}{n t_n}
\ge\sum_n\frac{n^{\beta-\kappa}}{n}\|\bbT^n\teta\|_\kappa^2
%\\&=
=
\calO(1)\sum_n \left(\sum_i\exp(-nq_i)(nq_i)^{\beta-\kappa-1} q_i\cdot\lambda_i^{-\beta}\theta_i^2\right)\\
&=
\calO(1)\sum_i\left(\sum_n \exp(-nq_i)(nq_i)^{\beta-\kappa-1} q_i\right)\cdot\lambda_i^{-\beta}\theta_i^2
%\\&
\ge \calO(1)\sum_i \Gamma(\beta-\kappa)\cdot\lambda_i^{-\beta}\theta_i^2\\
&=
\calO(1)\Gamma(\beta-\kappa)\cdot\sum_i\lambda_i^{-\beta}\theta_i^2=
\calO(1)\Gamma(\beta-\kappa)\|\teta\|_\beta^2\,.
\end{align*}
where we approximated the series by the integral as in Lemma~\ref{lem:series and Gamma}
and changed the variables in the integral. 
\end{proof}

%%%%%%%%%%%%%%%%%Of the Theorem on averages
%%%%%%%%%%%%%%%%%%%%%%%%%%%%%%%%%%%%%%%%%%%%%%%%%%%%%%%%%%%%%%%%%

\begin{lemma}\label{lem:neutral recursion}
Let\  $0<a_n<1$ satisfies $a_{n+1}\le a_n-a_n^{1+w}$ for some $w>0$. Then 
$a_n\le a_0(1+n w a_0^w)^{-1/w}$. If $c_{n+1}\le c_n-Kc_{n}^{1+w}$ then $c_n\le c_0(1+n w K c_0^w)^{-1/w}$.
  (Proof: see Section~\ref{sec6:lem:neutral recursion}.)
\end{lemma}

\begin{lemma}[H\"older inequality for $\varphi$, see \cite{BerthialBachGaillard}]\label{lem: Holder}
Let $\beta<\kappa<\alpha$ and $p=\frac{\alpha-\kappa}{\alpha-\beta}$. 
Then 
\[
\varphi_\kappa\,\le\varphi_\beta^p\,\varphi_\alpha^{1-p}=
\varphi_\beta^{\frac{\alpha-\kappa}{\alpha-\beta}}\,\varphi_\alpha^{\frac{\kappa-\beta}{\alpha-\beta}}
\qquad\text{and}\qquad
    \varphi_\beta\ge \varphi_\kappa^{\frac{1}{p}}\varphi_\alpha^{1-\frac{1}{p}}
    =\varphi_\kappa^{1+\frac{\kappa-\beta}{\alpha-\kappa}}
    \varphi_\alpha^{-\frac{\kappa-\beta}{\alpha-\kappa}}.
\]
\end{lemma}
\begin{proof}
We have $\kappa=p\beta+(1-p)\alpha$ and $\varphi_\kappa(\teta)=\sum \lambda_i^\kappa\theta_i^2\le \sum \lambda_i^{p\beta+(1-p)\alpha} \theta_i^{2(p+(1-p))}
=\sum (\lambda_i^\beta \theta_i^2)^p(\lambda_i^\alpha \theta_i^2)^{1-p}
\le
(\sum \lambda_i^\beta \theta_i^2)^p\cdot(\sum \lambda_i^\alpha \theta_i^2)^{1-p}
=\varphi_\beta(\teta)^p\varphi_\alpha(\teta)^{1-p}\,.
$
\end{proof}
%\begin{coro}[Alternative H\"older inequality]With the notation of Lemma~\ref{lem: Holder}, 
%\begin{equation}
%    \label{eqn:altHolder}
%    \varphi_\beta\ge \varphi_\kappa^{\frac{1}{p}}\varphi_\alpha^{1-\frac{1}{p}}
%    =\varphi_\kappa^{1+\frac{\kappa-\beta}{\alpha-\kappa}}
%    \varphi_\alpha^{-\frac{\kappa-\beta}{\alpha-\kappa}}.
%\end{equation}
%\end{coro}

\begin{lemma}[Main recursion formula, see \cite{BerthialBachGaillard}]\label{lem:main recursion}
%With $\varphi_\beta=\scpr{\teta}{\bbS^{-\beta}\teta}$, %let 
%$\hat{\teta}=\bbTx\teta=\teta-\gamma\bbSx\teta$ %and %$\hat\varphi_\beta=\ex[\varphi_\beta\hat{\teta}]$ 
%then
\begin{equation}
    \label{eqn:main recursion}
    \ex[\varphi_\beta(\bbTx\teta)]=\varphi_\beta(\teta)-2\gamma\varphi_{\beta-1}(\teta)+\gamma^2\ex[\scpr{\teta}{\xx}^2\scpr{\xx}{\bbS^{-\beta}\xx}]\,. 
\end{equation}
  (Proof: see Section~\ref{sec6:lem:main recursion}.)
\end{lemma}

Another form of the last term of~\eqref{eqn:main recursion} is 
$    %\label{eqn:main recursion-alternative}
    \ex[\scpr{\teta}{\xx}^2\scpr{\xx}{\bbS^{-\beta}\xx}]=
    \ex[\scpr{\teta}{\bbSx\teta}\varphi_\beta(\xx)]$.
    
\begin{coro}\label{lem:teta(n) decreases}
Under~{\bf\ref{asm:main}}, if $\gamma<2/C_\beta$ then for any $\teta$, the sequence
$\ex[\scpr{\teta(n)}{\bbS^{-\beta}\teta(n)}]$ is decreasing in $n$ and thus bounded from above 
 by $M_\beta:=\ex[\scpr{\teta(0)}{\bbS^{-\beta}\teta(0)}]$ uniformly in $n$.

(Proof: see Section~\ref{sec6:lem:teta(n) decreases}.)
\end{coro}
%\section{Upper bound for the second moment}

\begin{prop}[Upper bound for the convergence of $\teta(n)$]\label{prop:upper bound rnd}
Under~{\bf\ref{asm:main}}, for any $0\le\kappa<\beta<\alpha(\teta)$, if $\gamma<2/C_\beta$ then we have 
\[
\ex[\scpr{\teta(n)}{\bbS^{-\kappa}\teta(n)}]\le \calO(1)n^{-(\beta-\kappa)}\,.
\]
\end{prop}

\begin{proof}

For any $\kappa<\beta<\alpha$ we have 
with $p=\frac{\beta-\kappa}{1+\beta-\kappa}\in(0,1)$ the convex combination $\kappa =p (\kappa-1)+(1-p)\beta$. By Lemma~\ref{lem: Holder} (H\"older inequality) we get\quad
$
    \ex\scpr{\teta}{\bbS^{-\kappa}\teta}\le
    \ex\scpr{\teta}{\bbS^{1-\kappa}\teta}^p\ex\scpr{\teta}{\bbS^{-\beta}\teta}^{1-p},
$
\quad from which it follows that \quad 
$
    \ex\scpr{\teta}{\bbS^{1-\kappa}\teta}
    \ge\ex\scpr{\teta}{\bbS^{-\kappa}\teta}^{1/p}
    \ex\scpr{\teta}{\bbS^{-\beta}\teta}^{1-1/p}.
$
We apply this to the sequence $\teta(n)$ and get 
\begin{align*}
    \ex\scpr{\teta(n)}{\bbS^{1-\kappa}\teta(n)}
    &\ge\ex\scpr{\teta(n)}{\bbS^{-\kappa}\teta(n)}^{1+\frac{1}{\beta-\kappa}}
    \ex\scpr{\teta(n)}{\bbS^{-\beta}\teta(n)}^{-\frac{1}{\beta-\kappa}}
    %\\
    %&
    \ge
    \ex\scpr{\teta(n)}{\bbS^{-\kappa}\teta(n)}^{1+\frac{1}{\beta-\kappa}}
    M_\beta^{-\frac{1}{\beta-\kappa}}\,,
\end{align*}
where in the last inequality we used Corollary~\ref{lem:teta(n) decreases}.
Setting $\ex\scpr{\teta(n)}{\bbS^{1-\kappa}\teta(n)}=\phi(n)$, 
$w=\frac{1}{\beta-\kappa}$ and $K=M_\beta^{-w}$
we get the recursion
$\phi_{n+1}(\kappa)\le \phi_n{\kappa}-K\phi_n(\kappa)^{1+w}$.
Now apply Lemma~\ref{lem:neutral recursion} and get
$\phi_n(\kappa)\le \calO(1)n^{-1/w}$, where the constant $\calO(1)$ may depend on $\kappa$ and $\beta$, but not on $n$.
\end{proof}
%{\color{blue}

\commentOUT{
\begin{rem}
\label{rem:two}
Some assumption like {\bf \ref{asm:main}} is essential for deriving the upper bound in Proposition \ref{prop:upper bound rnd}. 
Lack of regularity of the sample data may violate {\bf \ref{asm:main}}.
%\end{rem}
It is exemplified in Theorem 2 of~\cite{BerthialBachGaillard}, where the sample data satisfies  $\scpr{\xx}{\bbS^{-\beta}\xx}=\infty$, with positive probability. This condition violates {\bf Assumption~\ref{asm:Remark3}}, therefore {\bf \ref{asm:main}}, and denies the convergence rate of Theorem \ref{thm:upper bound}.  
Under a weaker condition $\ex[\scpr{\xx}{\bbS^{-\beta}\xx}]=\infty$, we could still have, $\ex[\scpr{\teta}{\xx}^2\scpr{\xx}{\bbS^{-\beta}\xx}]=\infty$. For example, assuming, without loss of generality, $\ex[\xx]=0$, in the case when $x_i$ and $x_j$ are independent for $i\neq j$, we have, 
$\ex[\scpr{\teta}{\xx}^2\scpr{\xx}{\bbS^{-\beta}\xx}]\ge \sum_{j} \ex[\theta^2_j x_j^2] \sum_i \g_i^{-\beta} \ex[x_i^2]=\scpr{\teta}{\bbS\teta}\ex[\scpr{\xx}{\bbS^{-\beta}\xx}]=\infty$. 
%from the fact that $\ex[\scpr{\xx}{\bbS^{-\beta}\xx}]=\sum_i \g_i^{-\beta} \ex[x_i^2]$. 
{\bf Assumption~\ref{asm:main}} is then also violated, and by Lemma \ref{lem: main recursion} the convergence rate is denied.  
\end{rem}
}

\section{Almost sure convergence}
\label{sec:asc}
Denote, 
%the following function on the unit sphere in $\bbH$,
\begin{align*}
%\label{eqn:defn_h}
h(z) := \ex[\scpr{z}{\xx}^2], \qquad
%\end{align*}
%where the expectation is taken with respect to the distribution of $\xx$, and
%\begin{align*}
M_n := \scpr{\frac{\teta_{n}}{\|\teta_{n}\|}}{\xx_{n+1}}^2-h\left(\frac{\teta_{n}}{\|\teta_{n}\|}\right).
\end{align*}
\begin{lemma}[Martingale] \label{lem:MDS}
Under the condition of $\ex[\|\xx\|^2]<\infty$, $M_n$ 
is a \emph{martingale difference sequence}. 
\end{lemma}
\begin{proof} By the definition of $h(\cdot)$, we have,
\begin{align*}
&\ex\left[M_n|\sigma(\xx_1, \xx_2, \ldots, \xx_{n-1})\right]= \ex\left[\scpr{\frac{\teta_{n-1}}{\|\teta_{n-1}\|}}{\xx_{n}}^2-h\left(\frac{\teta_{n-1}}{\|\teta_{n-1}\|}\right)|\sigma(\xx_1, \xx_2, \ldots, \xx_{n-1})\right]=0.
\end{align*}
\end{proof}

\begin{proof}[Proof of Theorem \ref{thm:asc} on almost sure convergence]
Recall that by \eqref{def:Sx Tx}, $\teta_{n+1} =\bbT_{\xx_{n+1}} \teta_n= \teta_n-\gamma \bbS_{\xx_{n+1}} \teta_n$ and by 
Lemma~\ref{lem:main recursion} (use recursion formula \eqref{eqn:main recursion} with $\beta=0$) we have: 
$\|\teta_{n+1}\|^2=\|\teta_n\|^2-\gamma (2-\gamma \|\xx_{n+1}\|^2) \scpr{\teta_n}{\xx_{n+1}}^2$.

Let us consider first the boundedness condition, $\|\xx\|^2\le M$ a.s. Then for  $\gamma<2/M$ the sequence $\|\teta_n\|^2$ is decreasing. Hence, $\|\teta_n\|$ converges almost surely, and as $\teta_n$  converges to zero in probability, the result follows.

Now let us look at the more general case when $\ex[\|\xx\|^4]<\infty$. From the recursion formula we further have, 
\begin{align*}
\|\teta_{n+1}\|^2 &= \|\teta_n\|^2 \left[1-2\gamma \scpr{\frac{\teta_n}{\|\teta_n\|}}{\xx_{n+1}}^2 
+ \gamma^2 \scpr{\frac{\teta_n}{\|\teta_n\|}}{\xx_{n+1}}^2 \|\xx_{n+1}\|^2\right] .
\end{align*}
Hence, we can write,
\begin{align*}
\|\teta_{N+1}\|^2 = \|\teta_0\|^2 \prod_{n=0}^N
 \left[1-2\gamma \scpr{\frac{\teta_n}{\|\teta_n\|}}{\xx_{n+1}}^2 
+ \gamma^2 \scpr{\frac{\teta_n}{\|\teta_n\|}}{\xx_{n+1}}^2 \|\xx_{n+1}\|^2\right] .
\end{align*}
%We shall prove that
%\begin{align*}
%\prod_{n=0}^N
% \left[1-2\gamma \scpr{\frac{\teta_n}{\|\teta_n\|}}{\xx_{n+1}}^2 
%+ \gamma^2 \scpr{\frac{\teta_n}{\|\teta_n\|}}{\xx_{n+1}}^2 \|\xx_{n+1}\|^2\right]
%\end{align*}
 %converges to zero as $N\to\infty$ almost surely. As
 In order to prove that $||\teta(n)||^2\to 0$ almost surely, using
\begin{align*}
&\log \left\{ \prod_{n=0}^N
 \left[1-2\gamma\scpr{\frac{\teta_n}{\|\teta_n\|}}{\xx_{n+1}}^2 
+ \gamma^2 \scpr{\frac{\teta_n}{\|\teta_n\|}}{\xx_{n+1}}^2 \|\xx_{n+1}\|^2\right] \right\}\\
& \le -\sum_{n=0}^N \left[2\gamma \scpr{\frac{\teta_n}{\|\teta_n\|}}{\xx_{n+1}}^2 
- \gamma^2 \scpr{\frac{\teta_n}{\|\teta_n\|}}{\xx_{n+1}}^2 \|\xx_{n+1}\|^2\right], 
\end{align*}
it is enough to prove that
\begin{align}
\label{eqn:three_terms}
\sum_{n=0}^N M_n +  \sum_{n=0}^N h\left(\frac{\teta_n}{\|\teta_n\|}\right)
- \frac{\gamma}{2}\sum_{n=0}^N \scpr{\frac{\teta_n}{\|\teta_n\|}}{\xx_{n+1}}^2 \|\xx_{n+1}\|^2,
\end{align}
tends to $+\infty$ almost surely.

By Lemma \ref{lem:MDS}, we know that $\sum_{n=0}^N M_n$ is a martingale, and by Doob's convergence theorem, it converges almost surely to a random variable with finite mean, see e.g.~\cite{hall2014martingale}. The second term in \eqref{eqn:three_terms} is lower bounded by $(N+1)\delta$, where $\delta=\inf_{||z||=1}h(z)>0$ by assumption. The third term is lower bounded by 
\begin{align*}
-\frac{\gamma }{2}\sum_{n=0}^N\|\xx_{n+1}\|^4 = -(N+1) \frac{\gamma}{2} \frac{1}{N+1}\sum_{n=0}^N\|\xx_{n+1}\|^4.
\end{align*}
By the assumption of the finiteness of the fourth moment, we know that, $\frac{1}{N+1}\sum_{n=0}^N\|\xx_{n+1}\|^4$ converges almost surely to $\ex\|\xx_{n+1}\|^4$. Hence, under the assumed condition that $\gamma<\delta/\ex[\|\xx\|^4]$, the desired result follows. 
\end{proof}

%%%%%%%%%%%%%%%%%%%%%%%%%%%%%%%%%%%%%%%%%%%%
\section{ Proofs of Technical Lemmata}\label{sec:technical lemmata}
%%%%%%%%%%%%%%%%%%%%%%%%%%%%%%%%%%%%%%%%%%%%
\begin{proofdelayed}[for the equality in  Example~\ref{example2}]\label{sec6:example2}
\begin{align*}
& \ex\left[\scpr{\teta}{\xx}^2\varphi_\beta(\xx)\right]=\ex\left[\left(\sum_i\theta_i x_i\right)\cdot\left(\sum_k \theta_i x_k\right)\cdot\left(\sum_j \lambda_j^{-\beta}x_j^2\right)\right]\\
&=
\sum_{i,k,j}\ex\left[\theta_i\theta_i\lambda_j^{-\beta}x_ix_kx_j^2\right]
=
\left\{\sum_{j\not=i\not=k\not=j}+\sum_{j=i\not=k}+\sum_{i\not=k=j}+\sum_{j\not=i=k}+\sum_{i=k=j}\right\}\ex\left[\dots\right] =
\\
&\text{which using independence, 0 mean and calculated moments is equal to}\\
&= 0+0+0+\sum_{j\not=i}\ex\left[\theta_i^2\lambda_j^{-\beta}x_i^2x_j^2\right]+\sum_{i}\ex\left[\theta_i^2\lambda_i^{-\beta}x_i^4\right]
\\
&=
\sum_{i,j}\ex\left[\theta_i^2\lambda_j^{-\beta}x_i^2x_j^2\right]
-\sum_i\ex\left[\theta_i^2 x_i^2\right]\ex\left[\lambda_i^{-\beta}x_i^2\right]
+\sum_{i}\ex\left[\theta_i^2\lambda_i^{-\beta}x_i^4\right]
\\
&=\left(\sum_{i}\theta_i^2\lambda_i\right)\cdot\left(\sum_j\lambda_j^{1-\beta}\right)
-\sum_i\left(\theta_i^2 \lambda_i\right)\left(\lambda_i^{-\beta}\lambda_i\right)
+\sum_{i}\theta_i^2\lambda_i^{-\beta}(1+\lambda_i)\lambda_i
\\
&=
\sum_{i}\theta_i^2\lambda_i\cdot K_\beta
+\sum_{i}\theta_i^2\lambda_i^{1-\beta}\,.
\end{align*}
\end{proofdelayed}

\begin{proofdelayed}[of Proposition~\ref{prop:beta kappa}]\label{sec6:prop:beta kappa}
Apply $\teta=\bbe_i$ from the ON basis of $\bbS$ and get $\ex[x_i^2\|\xx\|^2]\le C_0\lambda_i$.
After  summing up for $i\le N$ we get
\begin{align*}
\ex\left[\sum_{i\le N}x_i^2\right]^2\le
\ex\left[\left(\sum_{i\le N} x_i^2\right)^2\right]\le
\ex\left[\left(\sum_{i\le N} x_i^2\right)\|\xx\|^2\right]\le
C_{0}\sum_{i\le N}\lambda_i=
C_{0}\ex \left[\sum_{i\le N}x_i^2\right]\,.
\end{align*}
Thus for any $N$ we have $\ex \left[\sum_{i\le N}x_i^2\right]\le C_0$ and the Proposition follows by taking the limit $N\to\infty$.
\end{proofdelayed}

\begin{proofdelayed}[of Lemma~\ref{lem:f(lambda)}]\label{sec6:lem:f(lambda)}
The function $f$ is continuous, and for $\lambda>0$, $\lambda\not=1$ we have:
  $f'(\lambda)=f(\lambda)\cdot\frac{1}{(1-\lambda)\lambda}\cdot(-m\lambda+\tau(1-\lambda))$.
  Then, as $f(0)=f(1)=0$ and $f>0$, and the only local maximum is possible at $\lambda_*$ where the value is
  \[
  f(\lambda_*)=\left(1-\frac{\tau}{m+\tau}\right)^{\frac{m+\tau}{\tau}\cdot \tau}\cdot\left(1-\frac{\tau}{m+\tau}\right)^{-\tau}
  \cdot\left(\frac{\tau}{m+\tau}\right)^\tau.\\
  \]
As $1-\xx\le e^{-\xx}\le 1-(1-e^{-1})\xx$ for $0\le \xx\le 1$, we have, $1-y\ge e^{-(e/e-1)y}$ with $y=(1-e^{-1})\xx$.
With  $z=\frac{\tau}{m+\tau}$ in place of $\xx$ on one side we get $(1-z)^{\tau/z}\le e^{-\tau}$ 
and with the same $z$ in place of $y$ on the other side  
we get $(1-z)^{\tau/z}\ge e^{-\tau\cdot e/e-1}$.
For $1\le \lambda<2-\epsilon$ we observe that $f$ is increasing there and  $f(\lambda)\le f(2-\epsilon)\le (1-\epsilon)^m 2^\tau$ which, as $m\to\infty$, decreases to 0  faster than $m^{-\tau}$.
\end{proofdelayed}

\begin{proofdelayed}[of Lemma~\ref{lem:series and Gamma}]\label{sec6:lem:series and Gamma}
For $q=-\ln(1-\mu)$ with $0<\mu<\frac{1}{2}$we have $\mu<q< (2\ln 2)\mu$ so for the term of the series 
we have 
$e^{nq}(nq)^{\kappa}(\ln 4)^{-\kappa} 
< 
\exp(-n\ln(1-\mu))(n\mu)^{\kappa} 
< 
e^{nq}(nq)^{\kappa}(\ln 4)^{\kappa}$, where the bounds can be tightened if we know the sign of $\kappa-1$. Now we can estimate the series $\sum_n e^{-qn}(qn)^{\kappa-1}q$ by the integral 
$\int_0^\infty e^{-qn}(qn)^{\kappa-1}\,d(qn)=\Gamma(\kappa)$ (use the variable $t=qn$). If $\kappa\le 1$ then the function to integrate is monotone and the comparison is standard. For $\kappa>1$ the function has a maximum at $\kappa-1$, and some care needs to be taken around this point. Luckily the values of the function for neighboring $n$'s are comparable:  
\[
\frac{e^{-q(n\pm1)}(q(n\pm1))^{\kappa-1}q}{e^{-qn}(qn)^{\kappa-1}q}
=e^{\mp q}\left(1\pm\frac{1}{n}\right)^{\kappa-1},
\]
which is bounded from above and below for bounded $q$, and $\kappa$, even near the  maximum $qn\approx\kappa-1$.
So that there exists $K>0$, such that, for every $n>0$, 
\[
\displaystyle{K\le \frac{e^{-qn}(qn)^{\kappa-1}q}{\int_{n-1}^n e^{-qm}(qm)^{\kappa-1}\,qd(m)}\le \frac{1}{K}\,.}
\]
\end{proofdelayed}

\begin{proofdelayed}[of Lemma~\ref{lem:neutral recursion}]\label{sec6:lem:neutral recursion}
The sequence is decreasing and the only accumulation point is 0. Let $a=b^{-1/w}$ with $b>1$ 
then
$b_{n+1}\ge b_n(1-1/b_n)^{-w}\ge b_n(1+1/b_n)^{w} \ge b_n(1+w/b_n)=b_n+w$ so that  $b_n\ge b_0+n w$ and $a_n\le (a_0^{-w}+n w)^{-1/w}$. Use this next for $a_n=K^{1/w}c_n$.
\end{proofdelayed}

\begin{proofdelayed}[of Lemma~\ref{lem:main recursion}]\label{sec6:lem:main recursion}
\begin{align*}
    \ex[\varphi_\beta(\bbTx\teta)]%\hat{\phi}_\beta(\teta)&=\varphi_\beta(\hat{\teta})=
    %\ex[\scpr{\hat{\teta}}{\bbS^{-\beta}\hat{\teta}}]
    %\\
    &=\ex[\scpr{\teta-\gamma\bbSx\teta}{\bbS^{-\beta}(\teta-\gamma\bbSx\teta)}
    =\ex[\scpr{\teta-\gamma\bbSx\teta}{\bbS^{-\beta}\teta-\gamma\bbS^{-\beta}\bbSx\teta}
    \\
    &=\ex[\scpr{{\teta}}{\bbS^{-\beta}{\teta}}]
    -\gamma\ex[\scpr{\teta}{\bbS^{-\beta}\bbSx\teta}]-\gamma\ex[\scpr{\bbSx\teta}{\bbS^{-\beta}\teta}]
    +\gamma^2\ex[\scpr{\bbSx\teta}{\bbS^{-\beta}\bbSx\teta}]
    \\
    &=
    \varphi_\beta(\teta)
    -\gamma\scpr{\teta}{\bbS^{-\beta}\ex[\bbSx\teta]}
    -\gamma\scpr{\ex[\bbSx\teta]}{\bbS^{-\beta}\teta}
        +\gamma^2\ex[\scpr{\scpr{\teta}{\xx}\xx}{\bbS^{-\beta}\scpr{\teta}{\xx}\xx}]
    \\
    &=\varphi_\beta(\teta)
    -\gamma\scpr{\teta}{\bbS^{-\beta}\bbS\teta}
    -\gamma\scpr{\bbS\teta}{\bbS^{-\beta}\teta}
        +\gamma^2\ex[\scpr{\teta}{\xx}^2\scpr{\xx}{\bbS^{-\beta}\xx}]
\\
    &=\varphi_\beta(\teta)
    -2\gamma\scpr{\teta}{\bbS^{-\beta+1}\teta}
        +\gamma^2\ex[\scpr{\teta}{\xx}^2\scpr{\xx}{\bbS^{-\beta}\xx}]\,.
\end{align*}
where we used the definition of $\bbS=\ex[\bbSx]$, the symmetry (Lemma~\ref{lem:symm and nonneg}) and commutativity of the powers of $\bbS$, 
$\bbS^{1}\circ\bbS^{-\beta}=\bbS^{1-\beta}=\bbS^{-\beta}\circ\bbS^1$, whenever well defined. 
\end{proofdelayed}

\begin{proofdelayed}[of Corollary~\ref{lem:teta(n) decreases}]\label{sec6:lem:teta(n) decreases}
By~{\bf\ref{asm:main}} there exists a constant $C_\beta$ such that for all $\teta\in\bbH$ we have,
$\ex[\scpr{\teta}{\xx}^2 \scpr{\teta}{\bbS^{-\beta}\teta}]\le C_\beta\ex[\scpr{\teta}{\bbS^{1-\beta}\teta}]$. Hence
by Lemma~\ref{lem:main recursion},
\begin{align*}
\ex[\scpr{\teta(n+1)}{\bbS^{-\beta}\teta(n+1)}]&\le \ex[\scpr{\teta(n)}{\bbS^{-\beta}\teta(n)}] 
-2\gamma \ex[\scpr{\teta(n)}{\bbS^{1-\beta}\teta(n)}] 
+ \gamma^2 C_\beta \ex[\scpr{\teta}{\bbS^{1-\beta}\teta}]\\
&=\ex[\scpr{\teta(n)}{\bbS^{-\beta}\teta(n)}] 
-\gamma(2-\gamma C_\beta) \ex[\scpr{\teta(n)}{\bbS^{1-\beta}\teta(n)}] 
\end{align*}
and as $\ex[\scpr{\teta}{\bbS^{1-\beta}\teta}]> 0$ for $\gamma<2/C_\beta$, the (positive) quantity 
$\ex[\scpr{\teta(n)}{\bbS^{-\beta}\teta(n)}]$ is decreasing in $n$, 
and therefore bounded from above for all $n$ by $\scpr{\teta(0)}{\bbS^{-\beta}\teta(0)}$.
\end{proofdelayed}

\commentOUT{\section{noise?}}



%% References


\bibliographystyle{elsarticle-num}
\bibliography{hmc}





% \end{thebibliography}

\end{document}